\documentclass{amsart}
\usepackage{mathptmx, amssymb, enumerate, colonequals, url, color}
\definecolor{chianti}{rgb}{0.6,0,0}
\definecolor{meretale}{rgb}{0,0,.6}
\definecolor{leaf}{rgb}{0,.35,0}
\usepackage[colorlinks=true, pagebackref, hyperindex, citecolor=meretale, urlcolor=leaf, linkcolor=chianti]{hyperref}

\newtheorem{theorem}{Theorem}[section]
\newtheorem{lemma}[theorem]{Lemma}
\newtheorem{proposition}[theorem]{Proposition}

\theoremstyle{definition}
\newtheorem{example}[theorem]{Example}
\newtheorem{remark}[theorem]{Remark}
\numberwithin{equation}{theorem}

\def\ge{\geqslant}
\def\le{\leqslant}
\def\to{\longrightarrow}
\def\mapsto{\longmapsto}
\def\tilde{\widetilde}

\def\GL{\operatorname{GL}}
\def\SL{\operatorname{SL}}

\def\Frac{\operatorname{frac}}
\def\Reflex{\operatorname{Ref}}
\def\rank{\operatorname{rank}}
\def\Hom{\operatorname{Hom}}
\def\grmod{\operatorname{grmod}}
\def\Qgrmod{\QQ\grmod}
\def\Sym{\operatorname{Sym}}
\def\LD{\operatorname{LD}}
\def\LRep{\operatorname{LRep}}
\def\LInd{\operatorname{LInd}}

\renewcommand{\mod}{\operatorname{mod}}

\def\fraka{\mathfrak{a}}
\def\frakm{\mathfrak{m}}
\def\frakS{\mathfrak{S}}

\def\FF{\mathbb{F}}
\def\FFp{\mathbb{F}_{\!p}}
\def\NN{\mathbb{N}}
\def\QQ{\mathbb{Q}}
\def\ZZ{\mathbb{Z}}

\def\calS{\mathcal{S}}

\def\id{\mathrm{id}}
\def\tr{\mathrm{tr}}

\begin{document}
\title[Frobenius representation type for invariant rings of finite groups]{Frobenius representation type for invariant rings\\
of finite groups}

\author{Mitsuyasu Hashimoto}
\address{Department of Mathematics, Osaka Metropolitan University, Sumiyoshi-ku, Osaka 558--8585, JAPAN}
\email{mh7@omu.ac.jp}

\author{Anurag K. Singh}
\address{Department of Mathematics, University of Utah, 155 South 1400 East, Salt Lake City, UT~84112, USA}
\email{singh@math.utah.edu}

\thanks{M.H. was partially supported by JSPS KAKENHI Grant number 20K03538 and MEXT Promotion of Distinctive Joint Research Center Program JPMXP0723833165; A.K.S.~was supported by NSF grants DMS~2101671 and DMS~2349623.}


\dedicatory{Dedicated to Professor Kei-ichi Watanabe, in celebration of his 80th birthday}

\begin{abstract}
Let $V$ be a finite rank vector space over a perfect field of characteristic $p>0$, and let $G$ be a finite subgroup of $\operatorname{GL}(V)$. If $V$ is a permutation representation of $G$, or more generally a monomial representation, we prove that the ring of invariants $(\operatorname{Sym}V)^G$ has finite Frobenius representation type. We also construct an example with $V$ a finite rank vector space over the algebraic closure of the function field ${\mathbb{F}_3}(t)$, and $G$ an elementary abelian subgroup of $\operatorname{GL}(V)$, such that the invariant ring $(\operatorname{Sym}V)^G$ does not have finite Frobenius representation type.
\end{abstract}
\maketitle

\section{Introduction}
\label{section:introduction}

The study of rings of finite Frobenius representation type was initiated by Smith and Van den Bergh~\cite{Smith:VdB}, as part of an attack on the conjectured simplicity of rings of differential operators on invariant rings; indeed, using this notion, they proved that if $R$ is a graded direct summand of a polynomial ring over a perfect field $k$ of positive characteristic, e.g., if $R$ is the ring of invariants for a linearly reductive group acting linearly on the polynomial ring, then the ring of $k$-linear differential operators on $R$ is a simple ring~\cite[Theorem~1.3]{Smith:VdB}. 

A reduced ring $R$ of prime characteristic $p>0$, satisfying the Krull-Schmidt theorem, has \emph{finite Frobenius representation type} (FFRT) if there exists a finite set~$\calS$ of~$R$-modules such that for each integer $e\ge0$, each indecomposable $R$-module summand of~$R^{1/p^e}$ is isomorphic to an element of $\calS$; the FFRT property and its variations are reviewed in~\S\ref{section:preliminaries}. Examples of rings with FFRT include Cohen-Macaulay rings of finite representation type, graded direct summands of polynomial rings~\cite[Proposition~3.1.6]{Smith:VdB}, and Stanley-Reisner rings~\cite[Example~2.3.6]{Kamoi}. More recently, Raedschelders, \v Spenko, and Van den Bergh proved that over an algebraically closed field of characteristic $p\ge\max\{n-2,3\}$, the Pl\"ucker homogeneous coordinate ring of the Grassmannian~$G(2,n)$ has FFRT~\cite{RSV}. In another direction, work of Hara and Ohkawa~\cite{Hara:Ohkawa} investigates the FFRT property for two-dimensional normal graded rings in terms of $\QQ$-divisors.

In addition to the original motivation, the FFRT property has found several applications. Suppose a ring $R$ has FFRT. Then Hilbert-Kunz multiplicities over $R$ are rational numbers by~\cite{Seibert}; tight closure and localization commute in $R$,~\cite{Yao}; local cohomology modules of the form~$H^k_\fraka(R)$ have finitely many associated primes,~\cite{Takagi:Takahashi, Hochster:Nunez, Dao:Quy}. For more on the~FFRT property, we point the reader towards \cite{Alhazmy:Katzman, Kamoi, Mallory:ffrt, Shibuta1, Shibuta2, Singh:Watanabe}.

Our goal here is to investigate the FFRT property for rings of invariants of finite groups. Let $V$ be a finite rank vector space over a perfect field $k$ of characteristic~$p>0$, and let $G$ be a finite subgroup of $\GL(V)$. In the nonmodular case, that is, when the order of~$G$ is not divisible by $p$, the invariant ring $S^G$ is a direct summand of the polynomial ring~$S\colonequals\Sym V$ via the Reynolds operator; it follows by \cite[Proposition~3.1.4]{Smith:VdB} that~$S^G$ has FFRT. The question becomes more interesting in the modular case, i.e., when $p$ divides $|G|$. We prove that if $V$ is a monomial representation of $G$, then the ring of invariants $S^G$ has~FFRT, Theorem~\ref{theorem:monomial}; this includes the case of a subgroup $G$ of the symmetric group~$\frakS_n$, acting on a polynomial ring $S\colonequals k[x_1,\dots,x_n]$ by permuting the indeterminates. On the other hand, while it had been expected that rings of invariants of reductive groups have FFRT (see for example the abstract of~\cite{RSV}), we prove that this is not the case:

\begin{theorem}
\label{theorem:not:FFRT}
Set $k$ to be the algebraic closure of the function field $\FF_3(t)$. Then there is an order $9$ subgroup $G$ of $\GL_3(k)$, such that $k[x_1,x_2,x_3]^G$ does not have FFRT.
\end{theorem}

This is proved as Theorem~\ref{theorem:counterexample}; the reader will find that a similar construction may be performed over any algebraically closed field $k$ that is not algebraic over $\FFp$. However, we do not know if $(\Sym V)^G$ always has~FFRT when $V$ is a finite rank vector space over $\overline{\FF}_{\!p}$, the algebraic closure of $\FFp$.

Returning to the nonmodular case, let $k$ be an algebraically closed field of characteristic~$p>0$, and $V$ a finite rank $k$-vector space. Set~$S\colonequals\Sym V$ and $R\colonequals S^G$, for $G$ a finite subgroup of $\GL(V)$ of order coprime to $p$. The rings~$S^{1/q}$ and~$R^{1/q}$ admit $\QQ$-gradings extending the standard $\NN$-grading on the polynomial ring $S$. Let~$M$ be a $\QQ$-graded finitely generated indecomposable $R$-module. By~\cite[Proposition~3.2.1]{Smith:VdB}, the module $M(d)$ is a direct summand of~$R^{1/q}$ for some $d\in\QQ$ if and only if
\[
M\ \cong\ (S\otimes_k L)^G
\]
for some irreducible representation $L$ of $G$. Let $V_1,\dots,V_\ell$ be a complete set of representatives of the isomorphism classes of irreducible representations of $G$, and set
\[
M_i\colonequals(S\otimes_k V_i)^G
\]
for $i=1,\dots,\ell$. Then, for each integer $e\ge 0$, the decomposition of $R^{1/p^e}$ into indecomposable $R$-modules takes the form
\[
R^{1/p^e}\ \cong\ \bigoplus_{i=1}^\ell\bigoplus_{j=1}^{c_{ie}}M_i(d_{ij}),
\]
where $d_{ij}\in\QQ$ and $c_{ie}\in\NN$. Suppose additionally that $G$ does not contain any pseudo\-reflections; by \cite[Theorem~3.4]{Hashimoto:Nakajima}, the \emph{generalized $F$-signature}
\[
s(R,M_i)\colonequals\lim_{e\to\infty}\frac{c_{ie}}{p^{e(\dim R)}}
\]
then agrees with
\[
(\rank_k V_i)/|G|.
\]
By~\cite[Theorem~5.1]{Hashimoto:Symonds}, this description of the asymptotic behavior of~$R^{1/p^e}$ remains valid in the modular case. It follows that for the invariant ring $R\colonequals k[x_1,x_2,x_3]^G$ in Theorem~\ref{theorem:not:FFRT}, while there exist infinitely many nonisomorphic indecomposable $R$-modules that are direct summands of some~$R^{1/p^e}$ up to a degree shift, almost all are \lq\lq asymptotically negligible.\rq\rq\

In \S\ref{section:preliminaries}, we review some basics on the FFRT property and on equivariant modules; these are used in \S\ref{section:counterexample} in the proof of Theorem~\ref{theorem:not:FFRT}. In \S\ref{section:monomial}, we prove that if $V$ is a monomial representation then $(\Sym V)^G$ has FFRT, and also that $(\Sym V)^G$ is $F$-pure in this case; the latter extends a result of Hochster and Huneke~\cite[page~77]{HH:Annals} that $(\Sym V)^G$ is $F$-pure when $V$ is a permutation representation. Lastly, in \S\ref{section:not:fpure}, we construct a family of examples that are not $F$-regular or $F$-pure, but nonetheless have the FFRT property.

\section{Preliminaries}
\label{section:preliminaries}
We collect some definitions and results that are used in the sequel.

\subsection*{Krull-Schmidt category}
Let $k$ be a perfect field of characteristic $p>0$, and $R$ a finitely generated \emph{positively graded} commutative $k$-algebra, i.e., $R$ is $\NN$-graded with~${[R]}_0=k$. Let $R\Qgrmod$ denote the category of finitely generated $\QQ$-graded $R$-modules. For modules~$M,N$ in $R\Qgrmod$, the module $\Hom_R(M,N)$ again lies in $R\Qgrmod$; in particular,
\[
\Hom_{R\Qgrmod}(M,N)\ =\ {[{\Hom_R(M,N)}]}_0
\]
is a finite rank $k$-vector space. Since $\Hom_{R\Qgrmod}(M,M)={[{\Hom_R(M,M)}]}_0$ has finite rank for each $M$ in $R\Qgrmod$, the category $R\Qgrmod$ is Krull-Schmidt; see \cite[\S3]{Hashimoto:Yang}.

\subsection*{Frobenius twist}
Let $e$ be a nonnegative integer. For a $k$-vector space $V$, we use ${}^eV$ to denote the $k$-vector space that coincides with $V$ as an abelian group, but has the left $k$-action~$\alpha\cdot v=\alpha^{p^e}v$ for~$\alpha\in k$ and $v\in V$, with the right action unchanged. An element~$v\in V$, when viewed as an element of ${}^eV$, will be denoted ${}^ev$, so
\[
{}^eV\ =\ \{{}^ev\mid v\in V\}.
\]
The map $v\mapsto {}^ev$ is an isomorphism of abelian groups, but not an isomorphism of $k$-vector spaces in general. Note that $\alpha\cdot {}^ev={}^e(\alpha^{p^e}v)$. When $V$ is $\QQ$-graded, we define a $\QQ$-grading on ${}^eV$ as follows: for a homogeneous element $v\in V$, set
\[
\deg{}^ev\colonequals(\deg v)/p^e.
\]

Let $V$ and $W$ be $k$-vector spaces. For $f\in\Hom_k(V,W)$, we define ${}^e\!f\colon{}^eV\to{}^eW$ by~${}^e\!f({}^ev)={}^e(fv)$. It is easy to see that ${}^e\!f\in\Hom_k({}^eV,{}^eW)$. This makes ${}^e(-)$ an auto-equivalence of the category of $k$-vector spaces. Note that the map
\[
{}^eV\otimes_k\!{}^eW\to{}^e(V\otimes_k W)
\]
with ${}^ev\otimes{}^ew\mapsto{}^e(v\otimes w)$ is well-defined, and an isomorphism. It is easy to check that ${}^e(-)$ is a monoidal functor; the composition ${}^e(-)\circ{}^{e'}(-)$ is canonically isomorphic to ${}^{e+e'}(-)$, and ${}^0(-)$ is the identity.

For a $k$-vector space $V$, the map ${}^e(-)\colon\GL(V)\to\GL({}^eV)$ given by $f\mapsto{}^e\!f$ is an isomorphism of abstract groups. If $V$ is a $G$-module, then the composition
\[
G\to\GL(V)\to\GL({}^eV)
\]
gives ${}^eV$ a $G$-module structure. Thus, $g({}^ev)={}^e(gv)$ for $g\in G$ and $v\in V$. Suppose $x_1,\dots,x_n$ is a $k$-basis of $V$. Then for each integer $e\ge 0$, the elements ${}^ex_1,\dots,{}^ex_n$ form a $k$-basis for~${}^eV$. If $f\in\GL(V)$ has matrix $(m_{ij})$ with respect to the basis $x_1,\dots,x_n$, then the matrix for ${}^e\!f$ with respect to ${}^ex_1,\dots, {}^ex_n$ is $(m_{ij}^{1/p^e})$.
Indeed,
\[
{}^e\!f({}^ex_j)\ =\ {}^e(fx_j)\ =\ {}^e(\sum_im_{ij}x_i)\ =\ \sum_i{}^e(m_{ij}x_i)\ =\ \sum_i m_{ij}^{1/p^e}\cdot {}^ex_i.
\]

When $R$ is a $k$-algebra, the $k$-algebra ${}^eR$ has multiplication defined by $({}^er)({}^es)\colonequals {}^e(rs)$. For $R$ a commutative $k$-algebra, the iterated Frobenius map $F^e\colon R\to {}^eR$ with
\[
r\mapsto {}^e(r^{p^e})
\]
is a homomorphism of~$k$-algebras. When $R$ is a positively graded finitely generated commutative $k$-algebra, the ring ${}^eR$ admits a $\QQ$-grading where for homogeneous $r\in R$,
\[
\deg{}^er\colonequals(\deg r)/p^e.
\]
The ring ${}^eR$ is then positively graded in the sense that ${[{}^eR]}_j=0$ for $j<0$, and ${[{}^eR]}_0=k$. The iterated Frobenius map $F^e\colon R\to {}^eR$ is degree-preserving and module-finite. Moreover,
\[
{}^e(-)\colon R\Qgrmod\to R\Qgrmod
\]
is an exact functor. If $M\in R\Qgrmod$, then the graded $k$-vector space ${}^eM$ is equipped with the~$R$-action~$r\cdot {}^em={}^e(r^{p^e}m)$, so ${}^eM$ is the graded ${}^eR$-module with the action ${}^er\cdot {}^em={}^e(rm)$, and the action of $R$ on ${}^eM$ is induced via $F^e\colon R\to{}^eR$.

When $R$ is reduced, it is sometimes more transparent to use the notation~$r^{1/p^e}$ in place of ${}^er$, and $R^{1/p^e}$ in place of ${}^eR$.

\subsection*{Graded FFRT}
When the equivalent conditions in the following lemma are satisfied, the ring $R$ is said to have finite Frobenius representation type (FFRT) in the graded sense:

\begin{lemma}
Let $R$ be a positively graded finitely generated commutative $k$-algebra. Then the following are equivalent:
\begin{enumerate}[\quad\rm(1)]
\item There exist $M_1,\dots,M_\ell\in R\Qgrmod$ such that for any $e\ge 1$, one has
\[
{}^eR\ \cong\ M_1^{\oplus c_{1e}}\oplus\dots\oplus M_\ell^{\oplus c_{\ell e}}
\]
as (non-graded) $R$-modules.

\item There exist $M_1,\dots,M_\ell\in R\Qgrmod$ such that for any $e\ge 1$, the $R$-module ${}^eR$ is isomorphic, as a $\QQ$-graded $R$-module, to a finite direct sum of copies of modules of the form $M_i(d)$ with $1\le i\le \ell$ and $d\in\QQ$.
\end{enumerate}
\end{lemma}

\begin{proof}
The direction (2)$\implies$(1) is obvious; we prove the converse. Fix $e\ge 1$. For a positive integer $c$, set $M^{\langle c\rangle}$ to be $M$ with the grading ${[M^{\langle c\rangle}]}_{cj}={[M]}_j$. Then $M^{\langle c\rangle}$ is a $\QQ$-graded module over the graded ring $R^{\langle c\rangle}$. Taking $c$ sufficiently divisible, we may assume that $R^{\langle c\rangle}$ is $p^e\ZZ$-graded and each $M_i^{\langle c\rangle}$ is $\ZZ$-graded. By \cite[Corollary~3.9]{Hashimoto:Yang}, ${}^eR^{\langle c\rangle}$ is isomorphic to a finite direct sum of modules of the form $(M_i^{\langle c\rangle})(d)$ with $1\le i\le \ell$ and~$d\in\ZZ$. It follows that ${}^eR$ is a finite direct sum of modules of the form $M_i(d/c)$.
\end{proof}

It follows from \cite[Corollary~3.9]{Hashimoto:Yang} that $R$ has~FFRT in the graded sense if and only if the $\frakm$-adic completion $\widehat{R}$ has FFRT, for~$\frakm$ the homogeneous maximal ideal of $R$.

\subsection*{Pseudoreflections}
Let $V$ be a finite rank $k$-vector space. An element $g\in\GL(V)$ is a \emph{pseudoreflection} if $\rank(1_V-g)=1$. Let $G$ be a finite group and $V$ a $G$-module. The action of $G$ on $V$ is \emph{small} if $\rho\colon G\to\GL(V)$ is injective, and $\rho(G)$ does not contain a pseudoreflection. If in addition $G\subset\GL(V)$, then $G$ is a \emph{small subgroup of $\GL(V)$}.

\subsection*{The twisted group algebra}

Let $V$ be a finite rank $k$-vector space. Let $G$ be a subgroup of $\GL(V)$, and set $S\colonequals\Sym V$. If~$x_1,\dots,x_n$ is a basis for $V$, then~$\Sym V=k[x_1,\dots,x_n]$ is a polynomial ring in $n$ variables. The action of $G$ on $V$ induces an action of $G$ on the polynomial ring $S$ by degree preserving $k$-algebra automorphisms.

We say that $M$ is a $\QQ$-graded $(G,S)$-module if $M$ is a $G$-module as well as a $\QQ$-graded $S$-module such that the underlying $k$-vector space structures agree, each graded component~${[M]}_i$ is a $G$-submodule of~$M$, and $g(sm)=(gs)(gm)$ for all $g\in G$, $s\in S$, and $m\in M$.

We recall the \emph{twisted group algebra} construction $S*G$ from \cite{Auslander:purity}. Set $S*G$ to be~$S\otimes_k kG$ as a $k$-vector space, with $kG$ the group algebra, and define
\[
(s\otimes g)(s'\otimes g')\colonequals s(gs')\otimes gg'.
\]
For $s\in S$ homogeneous, set the degree of $s\otimes g$ to be that of $s$; this gives~$S*G$ a graded $k$-algebra structure. A $\QQ$-graded $S*G$-module $M$ is a $\QQ$-graded $(G,S)$-module where
\[
sm\colonequals(s\otimes 1)m \qquad\text{ and }\qquad gm\colonequals(1\otimes g)m.
\]
Conversely, if $M$ is a $\QQ$-graded $(G,S)$-module, then $(s\otimes g)m\colonequals sgm$, gives $M$ the structure of a $\QQ$-graded $S*G$-module. Thus, a $\QQ$-graded $S*G$-module and a $\QQ$-graded $(G,S)$-module are one and the same thing. Similarly, a homogeneous (i.e., degree-preserving) map of $\QQ$-graded $(G,S)$-modules is precisely a homomorphism of graded $S*G$-modules.

With this setup, one has the following equivalence of categories:

\begin{lemma}
\label{lemma:equivalence}
Let $V$ be a finite rank $k$-vector space, and $G$ a small subgroup of $\GL(V)$. Set $S\colonequals\Sym V$ and $T\colonequals S*G$. Let $T\Qgrmod$ denote the category of finitely generated $\QQ$-graded left $T$-modules, and ${}^*\Reflex(G,S)$ denote the full subcategory of $T\Qgrmod$ consisting of those that are reflexive as $S$-modules; let ${}^*\Reflex S^G$ denote the full subcategory of~$S^G\Qgrmod$ consisting of modules that are reflexive as $S^G$-modules.

Then one has an equivalence of categories
\[
{}^*\Reflex(G,S)\to {}^*\Reflex S^G,\qquad\text{where}\qquad M\mapsto M^G,
\]
with quasi-inverse $N\to (N\otimes_{S^G}S)^{**}$, where $(-)^{*}\colonequals\Hom_S(-,S)$.
\end{lemma}

For the proof, see \cite[Lemma~2.6]{Hashimoto:Kobayahi}; an extension to group schemes may be found in~\cite{Hashimoto:III}. Note that under the functor displayed above, one has ${}^eS\mapsto ({}^eS)^G={}^e(S^G)$.

\section{An invariant ring without FFRT}
\label{section:counterexample}

We construct the counterexample promised in Theorem~\ref{theorem:not:FFRT}; more precisely, we prove:

\begin{theorem}
\label{theorem:counterexample}
Let $k$ be the algebraic closure of $\FF_3(t)$, the rational function field in one indeterminate over $\FF_3$. Let $G$ be the subgroup of $\GL(k^3)$ generated by the matrices
\[
\begin{bmatrix}
1 & 1 & 0 \\
0 & 1 & 1 \\
0 & 0 & 1
\end{bmatrix}
\qquad\text{ and }\qquad
\begin{bmatrix}
1 & t & 0 \\
0 & 1 & t \\
0 & 0 & 1
\end{bmatrix}.
\]
Then $G$ is isomorphic to $\ZZ/3\ZZ\times\ZZ/3\ZZ$. The invariant ring for the natural action of $G$ on the polynomial ring $\Sym(k^3)$ does not have FFRT.
\end{theorem}

\begin{lemma}\label{lemma:counterexample}
Let $k\colonequals\overline{\FF_3(t)}$ as above. Let $G=\ZZ/3\ZZ\times\ZZ/3\ZZ=\langle \sigma,\tau\rangle$, where $\sigma^3=\id=\tau^3$, and $\sigma\tau=\tau\sigma$. Then the group algebra $kG$ equals the commutative ring $k[a,b]/(a^3,b^3)$, where~$a\colonequals\sigma-1$ and $b\colonequals\tau-1$. For $\alpha\in k$, set $V(\alpha)$ to be $k^3$ with the $G$-action determined by the homomorphism $G\to\GL_3(k)$ with
\[
\sigma\mapsto\begin{bmatrix}
1 & 1 & 0 \\
0 & 1 & 1 \\
0 & 0 & 1
\end{bmatrix}
\qquad\text{ and }\qquad
\tau\mapsto\begin{bmatrix}
1 & \alpha & 0 \\
0 & 1 & \alpha \\
0 & 0 & 1
\end{bmatrix}.
\]
Then:
\begin{enumerate}[\quad\rm(1)]
\item\label{lemma:counterexample:1} If $\alpha\notin \FF_3$, then the action of $G$ on $V(\alpha)$ is small.
\item\label{lemma:counterexample:2} For $\alpha\neq\beta$ in $k$, the $G$-modules $V(\alpha)$ and $V(\beta)$ are nonisomorphic.
\item\label{lemma:counterexample:3} The Frobenius twist ${}^e(V(\alpha))$ is isomorphic to $V(\alpha^{1/3^e})$ as a $G$-module.
\item\label{lemma:counterexample:4} For each $\alpha\in k$, the $G$-module $V(\alpha)$ is indecomposable.
\end{enumerate}
\end{lemma}

\begin{proof}
Setting
\[
N\colonequals \begin{bmatrix}
0 & 1 & 0 \\
0 & 0 & 1 \\
0 & 0 & 0
\end{bmatrix}
\]
and taking $I$ to be the identity matrix, one has
\begin{multline*}
\sigma^i\tau^j\ =\ (I+N)^i(I+\alpha N)^j
\ =\ \left[I+iN+\binom{i}{2}N^2\right]\left[I+j\alpha N+\binom{j}{2}\alpha^2N^2\right] \\
=\ I + (i+j\alpha)N + \left[\binom{i}{2} + ij\alpha +\binom{j}{2}\alpha^2\right]N^2,
\end{multline*}
so $\sigma^i\tau^j-I$ has rank $2$ unless $\alpha\in\FF_3$ or $(i,j)=(0,0)$ in $\FF_3^2$. This proves~\eqref{lemma:counterexample:1}.

For~\eqref{lemma:counterexample:2}, note that the annihilators of $V(\alpha)$ and $V(\beta)$ are the ideals $(b-\alpha a)$ and $(b-\beta a)$ respectively in $kG=k[a,b]/(a^3,b^3)$. These ideals are distinct when $\alpha\neq\beta$.

The representation matrices for $\sigma$ and $\tau$ in $\GL({}^e(V(\alpha)))$ are
\[
{}^e(I+N)=I+N\qquad\text{ and }\qquad {}^e(I+\alpha N)=I+\alpha^{1/3^e}N
\]
respectively, so ${}^eV(\alpha)\cong V(\alpha^{1/3^e})$ as $G$-modules, proving~\eqref{lemma:counterexample:3}.

For~\eqref{lemma:counterexample:4}, note that $kG$ is an artinian local ring, so each nonzero $kG$-module has a nonzero socle. The socle of $V(\alpha)$ is spanned by the vector $(1,0,0)^\tr$, and hence has rank one. It follows that $V(\alpha)$ is an indecomposable $kG$-module.
\end{proof}

\begin{proof}[Proof of Theorem~\ref{theorem:counterexample}]
Set $S$ to be the polynomial ring $\Sym(k^3)$, and $T\colonequals S*G$. For $M$ a nonzero module in~$T\Qgrmod$, set
\[
\LD(M)\colonequals\min\{i\in\QQ\mid {[M]}_i\neq 0\}
\qquad\text{ and }\qquad
\LRep(M)\colonequals {[M]}_{\LD(M)},
\]
i.e., $\LRep(M)$ is the nonzero $\QQ$-graded component of $M$ of least degree. Note that for~$d$ a rational number, $\LRep(M(d))$ and $\LRep(M)$ are isomorphic as $G$-modules.

As $T\Qgrmod$ is Krull-Schmidt, there is a unique decomposition $M=N_1\oplus\cdots\oplus N_r$ of~$M$ into indecomposable objects. Setting $d\colonequals\LD(M)$, we have
\[
\LRep(M)\ =\ {[M]}_d\ =\ {[N_1]}_d\oplus\cdots\oplus{[N_r]}_d.
\]
Suppose $\LRep(M)$ is an indecomposable $G$-module. After a possible change of indices, we may assume that $\LRep(M)={[N_1]}_d$ and that ${[N_j]}_d=0$ for $j>1$. Note that, up to isomorphism, $N_1$ is the unique indecomposable direct summand of $M$ with~$\LD(N_1)=\LD(M)$. We define $\LInd(M)\colonequals N_1$. Note that we have $\LRep(N_1)\cong\LRep(M)$.

For $M$ as above, and $d\in\QQ$, define
\[
M_{\langle d\rangle}\colonequals \!\!\!\!\!\! \bigoplus_{i\equiv d\mod\ZZ} \!\!\! {[M]}_i,
\]
which is also an element of $T\Qgrmod$.

Since the degree $1/3^e$-component of ${}^eS$ is ${}^eV(t) = V(t^{1/3^e})$, one has
\[
\LRep\big({}^eS_{\langle 1/3^e\rangle}\big)\ =\ V(t^{1/3^e}),
\]
which is indecomposable by Lemma~\ref{lemma:counterexample}\,\eqref{lemma:counterexample:4}. The $G$-modules $V(t)$, $V(t^{1/3})$, $V(t^{1/3^2})$, $\dots$ are nonisomorphic by Lemma~\ref{lemma:counterexample}\,\eqref{lemma:counterexample:2}, so the isomorphism classes of the indecomposable $T$-modules
\[
\LInd\big(S_{\langle 1\rangle}\big),\quad \LInd\big({}^1S_{\langle 1/3\rangle}\big),\quad \LInd\big({}^2S_{\langle 1/3^2\rangle}\big),\quad \dots
\]
are distinct; specifically, any two of these indecomposable objects of $\Qgrmod T$ are nonisomorphic even after a degree shift. By Lemma~\ref{lemma:equivalence}, it follows that the indecomposable $\QQ$-graded $S^G$-modules
\[
\Big(\LInd\big(S_{\langle 1\rangle}\big)\Big)^G,\quad \Big(\LInd\big({}^1S_{\langle 1/3\rangle}\big)\Big)^G,\quad \Big(\LInd\big({}^2S_{\langle 1/3^2\rangle}\big)\Big)^G,\quad \dots
\]
are nonisomorphic. These occur as indecomposable summands of ${}^e(S^G)$ for $e\ge 1$, so the ring $S^G$ does not have FFRT.
\end{proof}

\begin{remark}
\label{remark:presentation}
For the interested reader, we give a presentation of the invariant ring $S^G$ in Theorem~\ref{theorem:counterexample}. This was obtained using \texttt{Magma}~\cite{Magma}, though one may verify all claims by hand, after the fact. Take $S\colonequals \Sym V$ to be the polynomial ring $k[x_1,x_2,x_3]$, where the indeterminates $x_1,x_2,x_3$ are viewed as the standard basis vectors in $V\colonequals k^3$. Then the invariant ring $S^G$ is generated by the polynomials
\begin{alignat*}3
f_1 & \colonequals x_1,\\
f_3 & \colonequals tx_1^2x_2 - (t+1)x_1^2x_3 - (t+1)x_1x_2^2 + x_2^3,\\
f_5 & \colonequals t(t-1)^2x_1^4x_3 + t(t^2+1)x_1^3x_2^2 - t(t+1)x_1^3x_2x_3 - (t+1)^2x_1^3x_3^2 - (t+1)(t-1)^2x_1^2x_2^3\\
& \qquad + (t+1)^2x_1^2x_2^2x_3 + x_1^2x_3^3 - (t-1)^2x_1x_2^4 - (t+1)x_1x_2^3x_3 - (t+1)x_2^5,\\
f_9 & \colonequals x_3(x_2+x_3)(x_1-x_2+x_3)(tx_2+x_3)(tx_1+x_2+tx_2+x_3)(x_1-tx_1-x_2+tx_2+x_3)\\
& \qquad \times (t^2x_1-tx_2+x_3)(t^2x_1-tx_1+x_2-tx_2+x_3)(x_1+tx_1+t^2x_1-x_2-tx_2+x_3),
\end{alignat*}
where $f_9$ is the product over the orbit of $x_3$. These four polynomials satisfy the relation
\begin{multline*}
t(t-1)^2(t^2+1)f_1^3f_3^4 - t^2(t-1)^2f_1^4f_3^2f_5 + (t^3+1)f_3^5 + (t^3+1)f_1f_3^3f_5 - f_1^6f_9 + f_5^3
\end{multline*}
that defines a normal hypersurface. Using this defining equation, one may see that $S^G$ is not $F$-pure. The defining equation also confirms that the $a$-invariant is $a(S^G)=-3$, as follows from~\cite[Theorem~3.6]{Hashimoto:a:inv} or \cite[Theorem~4.4]{GJS} since $G$ is a subgroup of $\SL(V)$ without pseudoreflections.
\end{remark}

\section{Ring of invariants of monomial actions}
\label{section:monomial}

Let $k$ be a field of positive characteristic, and let $G$ be a finite group. Consider a finite rank $k$-vector space $V$ that is a $G$-module. A $k$-basis $\Gamma$ of $V$ is a \emph{monomial basis} for the action of $G$ if for each $g\in G$ and $\gamma\in\Gamma$, one has $g\gamma\in k\gamma'$ for some $\gamma'\in\Gamma$. We say that~$V$ is a \emph{monomial representation} of $G$ if $V$ admits a monomial basis.

A monomial representation $V$ as above is a \emph{permutation representation} of $G$ if $V$ admits a $k$-basis $\Gamma$ such that each $g\in G$ permutes the elements of $\Gamma$.

\begin{theorem}
\label{theorem:monomial}
Let $k$ be a perfect field of positive characteristic, $G$ a finite group, and $V$ a monomial representation of $G$ over $k$. Then the ring of invariants $(\Sym V)^G$ has FFRT.
\end{theorem}

\begin{proof}
Set $q\colonequals p^e$, where $k$ has characteristic $p$ and $e\in\NN$. The action of $G$ on $S\colonequals\Sym V$ extends uniquely to an action of $G$ on~${}^eS=S^{1/q}$; note that
\[
(S^{1/q})^G\ =\ (S^G)^{1/q}.
\]
Let $\{x_1,\dots,x_n\}$ be a monomial basis for $V$. The ring $S^{1/q}$ then has an $S$-basis
\begin{equation}
\label{equation:basis}
B_e\colonequals\Big\{x_1^{\lambda_1/q}\cdots\, x_n^{\lambda_n/q}\ \mid\ \lambda_i\in\ZZ,\quad 0\le \lambda_i\le q-1\Big\}.
\end{equation}
For $\mu\in B_e$, set $\gamma_\mu$ to be the $k$-vector space spanned by the elements $g\mu$ for all $g\in G$. Then~$S^{1/q}$ is a direct sum of modules of the form $S\gamma_\mu$, and the action of $G$ on~$S^{1/q}$ restricts to an action on each $S\gamma_\mu$. To prove that $S^G$ has FFRT, it suffices to show that there are only finitely many isomorphism classes of $S^G$-modules of the form
\[
(S\gamma_\mu)^G\ =\ \Big(\sum_{g\in G}Sg\mu\Big)^G
\]
as $e$ varies. Fix $\mu\in B_e$, and consider the rank one $k$-vector space $k\mu$. Set
\[
H\colonequals\{g\in G\ \mid\ g\mu\in k\mu\}.
\]
Let $g_1,\dots,g_m$ be a set of left coset representatives for $G/H$, where $g_1$ is the group identity. We claim that the map
\begin{equation}
\label{equation:perm:map}
\sum_{i=1}^m g_i\colon (S\mu)^H\ \to\ (S\gamma_\mu)^G
\end{equation}
is an isomorphism of $\QQ$-graded $S^G$-modules. Assuming the claim, $(S\mu)^H=(S\otimes_k k\mu)^H$ is isomorphic, up to a degree shift, with a module of the form $(S\otimes_k\chi)^H$, where $\chi$ is a rank one representation of~$H$. Since there are only finitely many subgroups $H$ of~$G$, only finitely many rank one representations $\chi$ of $H$, and only finitely many isomorphism classes of indecomposable $\QQ$-graded $S^G$-summands of $(S\otimes_k\chi)^H$ by the Krull-Schmidt theorem, the claim indeed completes the proof.

It remains to verify the isomorphism~\eqref{equation:perm:map}. Given $g\in G$, there exists a permutation~$\sigma\in\frakS_m$ such that $gg_i=g_{\sigma i}h_i$ for each $i$, with $h_i\in H$. Given $s\mu\in (S\mu)^H$, one has
\[
g\Big(\sum_i g_i(s\mu)\Big)\ =\ \sum_i g_{\sigma i}h_i(s\mu)\ =\ \sum_i g_{\sigma i}(s\mu)\ =\ \sum_i g_i(s\mu),
\]
so $\sum_i g_i(s\mu)$ indeed lies in $(S\gamma_\mu)^G$. Since each $g_i$ is $S^G$-linear and preserves degrees, the same holds for their sum. As to the injectivity, if
\[
\sum_i g_i(s\mu)\ =\ \sum_i (g_is)(g_i\mu)\ =\ 0,
\]
then $g_is=0$ for each $i$, since $g_1\mu,\dots,g_m\mu$ are distinct elements of the basis $B_e$ as in~\eqref{equation:basis}, and hence linearly independent over $S$. But then $s=0$. For the surjectivity, first note that an element of $S\gamma_\mu$ may be written as $\sum_i s_ig_i\mu$. Consider
\[
f\colonequals s_1g_1\mu + s_2g_2\mu + \dots + s_mg_m\mu\ \in\ (S\gamma_\mu)^G.
\]
Apply $g_i$ to the above; since $g_if=f$, and $g_1\mu,\dots,g_m\mu$ are linearly independent over $S$, it follows that $g_is_1=s_i$. But then
\[
f\ =\ \sum_i g_i(s_1\mu),
\]
so it remains to show that $s_1\mu\in (S\mu)^H$. Fix $h\in H$. Since $hf=f$, one has
\[
\sum_i hg_i(s_1\mu) \ =\ \sum_i g_i(s_1\mu).
\]
As $hg_1\in H$ and $hg_i\notin H$ for $i\ge 2$, the linear independence of $g_1\mu,\dots,g_m\mu$ over $S$ implies that $h(s_1\mu)=s_1\mu$.
\end{proof}

\begin{remark}
\label{remark:f:pure}
For $k$ a field of positive characteristic, and $V$ a finite rank permutation representation of~$G$, Hochster and Huneke showed that the invariant ring~$(\Sym V)^G$ is $F$-pure~\cite[page~77]{HH:Annals}; the same holds more generally when $V$ is a monomial representation:

It suffices to prove the $F$-purity in the case where the field $k$ is perfect. With the notation as in the proof of Theorem~\ref{theorem:monomial}, $(S^G)^{1/q}$ is a direct sum of $S^G$-modules of the form $(S\gamma_\mu)^G$, where $\gamma_\mu$ is the $k$-vector space spanned by $g\mu$ for $g\in G$. When $\mu\colonequals 1$ one has $\gamma_\mu=k$, so~$S^G$ indeed splits from $(S^G)^{1/q}$.
\end{remark}

\begin{remark}
In Theorem~\ref{theorem:monomial} suppose, moreover, that $V$ is a permutation representation of~$G$. Then one may choose a basis $\{x_1,\dots,x_n\}$ for $V$ whose elements are permuted by~$G$. In this case, each $g\in G$ permutes the elements of $B_e$ for $e\in\NN$, and each rank one representation $\chi\colon H\to k^*$ is trivial; it follows that $(S^G)^{1/q}$ is a direct sum of $S^G$-modules of the form $S^H$, for subgroups~$H$ of $G$.
\end{remark}

\begin{example}
Let $p$ be a prime integer. Set $S\colonequals\FFp[x_1,\dots,x_p]$, and let $G\colonequals\langle\sigma\rangle$ be the cyclic group of order $p$ acting on $S$ by cyclically permuting the variables. The ring $S^G$ has~FFRT by Theorem~\ref{theorem:monomial}. Let $q=p^e$ be a varying power of $p$.

If $p=2$, then $S^G$ is a polynomial ring, and each $(S^G)^{1/q}$ is a free $S^G$-module; thus, up to isomorphism and degree shift, the only indecomposable summand of $(S^G)^{1/q}$ is $S^G$.

Suppose $p\ge 3$. For $\mu\in B_e$, consider the $kG$-submodule~$\gamma_\mu=kg\mu$ of $S^{1/q}$. If the stabilizer of $\mu$ is $G$, then $\gamma_\mu = k\mu$ is an indecomposable~$kG$ module, and $(S\mu)^G = S^G\mu \cong S^G$ is an indecomposable $S^G$-summand of $(S^G)^{1/q}$. Since the only subgroups of $G$ are $\{\id\}$ and~$G$, the only other possibility for the stabilizer of an element $\mu$ of $B_e$ is $\{\id\}$, in which case the orbit is a \emph{free orbit}, i.e., an orbit of size $|G|$, and $\gamma_\mu\cong kG$. We claim that
\[
(S\otimes_k kG)^G \cong S
\]
is an indecomposable $S^G$-module. Since the group $G$ contains no pseudoreflections in the case $p\ge 3$, Lemma~\ref{lemma:equivalence} is applicable, and it suffices to verify that $S\otimes_k kG$ is an indecomposable graded $(G,S)$-module. Note that~$kG=k[\sigma]/(1-\sigma)^p$ is an indecomposable $kG$-module. Suppose one has a decomposition as graded $(G,S)$-modules
\[
S \otimes_k kG\ \cong\ P_1 \oplus P_2,
\]
apply $(-)\otimes_S S/\frakm$ where $\frakm$ is the homogeneous maximal ideal of~$S$. Then
\[
kG\ \cong\ P_1/\frakm P_1 \oplus P_2/\frakm P_2.
\]
The indecomposability of $kG$ implies that $P_i/\frakm P_i=0$ for some $i$. But then Nakayama's lemma, in its graded form, gives $P_i=0$, which proves the claim. Lastly, it is easy to see that both of these types of $G$-orbits appear in $B_e$ if $e\ge 1$ so, up to isomorphism and degree shift, the indecomposable $S^G$-summands of $(S^G)^{1/q}$ are indeed $S^G$ and $S$.
\end{example}

\begin{example}
As a specific example of the above, consider the alternating group $A_3$ with its natural permutation action on the polynomial ring $S\colonequals \FF_3[x_1,x_2,x_3]$. For $q=3^e$, consider the $S$-basis~\eqref{equation:basis} for $S^{1/q}$. It is readily seen that the monomials
\[
(x_1x_2x_3)^{\lambda/q}\qquad\text{ where }\lambda\in\ZZ,\quad 0\le \lambda\le q-1
\]
are fixed by $A_3$, whereas every other monomial in $B_e$ has a free orbit. It follows that, ignoring the grading, the decomposition of ${(S^{A_3})}^{1/q}$ into indecomposable $S^{A_3}$-modules is
\[
{(S^{A_3})}^{1/q}\ \cong\ {(S^{A_3})}^q\oplus S^{(q^3-q)/3}.
\]
\end{example}

\begin{example}
Let $k$ be a perfect field of characteristic $2$ that contains a primitive third root~$\omega$ of unity. Let $G$ be the group generated by
\[
\sigma\colonequals\begin{bmatrix}
\omega & 0 \\
0 & \omega
\end{bmatrix}
\]
acting on $S\colonequals k[x_1,x_2]$. The invariant ring $S^G$ is the Veronese subring
\[
{k[x_1,x_2]}^{(3)}\ =\ k[x_1^3,\ x_1^2x_2,\ x_1x_2^2,\ x_2^3].
\]
The action of $G$ on $S$ extends to an action on $S^{1/q}$ where $\sigma(x_i^{1/q})=\omega^q x_i^{1/q}$. For $B_e$ as in~\eqref{equation:basis}, consider
\[
S^{1/q}\ =\ \bigoplus_{\mu\in B_e} S\mu.
\]
Suppose $\mu=x_1^{\lambda_1/q}x_2^{\lambda_2/q}$, where $\lambda_i$ are integers with $0\le \lambda_i\le q-1$. Then
\[
(S\mu)^G\ =\ \begin{cases}
S^G\mu & \text{if}\quad\lambda_1+\lambda_2\equiv 0\ \mod 3,\\
S^G x_1\mu + S^G x_2\mu & \text{if}\quad\lambda_1+\lambda_2\equiv 2q\ \mod 3,\\
S^G x_1^2\mu + S^G x_1x_2\mu + S^G x_2^2\mu & \text{if}\quad\lambda_1+\lambda_2\equiv q\ \mod 3.
\end{cases}
\]
The $S^G$-modules that occur in the three cases above are, respectively, isomorphic to the ideals $S^G$, $(x_1^3,x_1^2x_2)S^G$, and~$(x_1^3,x_1^2x_2,x_1x_2^2)S^G$, that constitute the indecomposable summands of $S^{1/q}$. The number of copies of each of these is \emph{asymptotically} $q^2/3$.

This extends readily to Veronese subrings of the form $k[x_1,x_2]^{(n)}$, for $k$ a perfect field of characteristic $p$ that contains a primitive $n$th root of unity; see \cite[Example~17]{Huneke:Leuschke}.
\end{example}

\begin{example}
Let $G\colonequals\langle\sigma\rangle$ be a cyclic group of order $4$, acting on $S\colonequals\FF_2[x_1,x_2,x_3,x_4]$ by cyclically permuting the variables. In view of~\cite{Bertin}, the invariant ring $S^G$ is a UFD that is not Cohen-Macaulay; $S^G$ has FFRT by Theorem~\ref{theorem:monomial}.

We describe the indecomposable summands that occur in an $S^G$-module decomposition of $(S^G)^{1/q}$ for $q=2^e$. The group $G$ contains no pseudoreflections, so Lemma~\ref{lemma:equivalence} applies. Consider the $S$-basis $B_e$ for $S^{1/q}$, as in~\eqref{equation:basis}. The monomials
\[
(x_1x_2x_3x_4)^{\lambda/q}\qquad\text{ where }0\le \lambda\le q-1
\]
are fixed by $G$; each such monomial $\mu$ gives an indecomposable~$kG$ module $\gamma_\mu = k\mu$, and an indecomposable $S^G$-summand $(S\mu)^G \cong S^G$ of $(S^G)^{1/q}$. The monomials $\mu$ of the form
\[
(x_1x_3)^{\lambda_1/q}(x_2x_4)^{\lambda_2/q}\qquad\text{ with }0\le \lambda_i\le q-1, \quad \lambda_1\neq\lambda_2
\]
have stabilizer $H\colonequals\langle\sigma^2\rangle$. In this case, $\gamma_\mu\cong k[\sigma]/(1-\sigma)^2$ is an indecomposable~$kG$ module, corresponding to an indecomposable $S^G$-summand $(S\otimes_k \gamma_\mu)^G\cong S^H$. Any other monomial in $B_e$ has a free orbit that corresponds to a copy of $(S\otimes_k kG)^G\cong S$.

Ignoring the grading, the decomposition of ${(S^G)}^{1/q}$ into indecomposable $S^G$-modules is
\[
{(S^G)}^{1/q}\ \cong\ {(S^G)}^q \oplus {(S^H)}^{(q^2-q)/2} \oplus S^{(q^4-q^2)/4}.
\]

\end{example}

\section{Examples that are FFRT but not \texorpdfstring{$F$}{F}-regular}
\label{section:not:fpure}

The notion of $F$-regular rings is central to Hochster and Huneke's theory of tight closure, introduced in~\cite{HH:JAMS}; while there are different notions of $F$-regularity, they coincide in the graded case under consideration here by~\cite[Corollary~4.3]{Lyubeznik:Smith}, so we downplay the distinction. The~FFRT property and $F$-regularity are intimately related, though neither implies the other: The hypersurface
\[
\FFp[x,y,z]/(x^2+y^3+z^5)
\]
has FFRT for each prime integer $p$, though it is not $F$-regular if $p\in\{2,3,5\}$; Stanley-Reisner rings have FFRT by~\cite[Example~2.3.6]{Kamoi}, though they are $F$-regular only if they are polynomial rings. For the other direction, the hypersurface
\[
R\colonequals\FFp[s,t,u,v,w,x,y,z]/(su^2x^2 + sv^2y^2 + tuvxy + tw^2z^2)
\]
is $F$-regular for each prime integer $p$, but admits a local cohomology module $H^3_{(x,y,z)}(R)$ with infinitely many associated prime ideals, \cite[Theorem~5.1]{Singh:Swanson}, and hence does not have~FFRT by~\cite[Corollary~3.3]{Takagi:Takahashi} or~\cite[Theorem~1.2]{Hochster:Nunez}. Nonetheless, for the invariant rings of finite groups that are our focus here, $F$-regularity implies FFRT; this follows readily from well-known results, but is recorded here for the convenience of the reader:

\begin{proposition}
Let $k$ be a perfect field, $G$ a finite group, and $V$ a finite rank $k$-vector space that is a $G$-module. If the invariant ring $(\Sym V)^G$ is $F$-regular, then it has FFRT.
\end{proposition}

\begin{proof}
An $F$-regular ring is \emph{splinter} by \cite[Theorem~5.25]{HH:JAG}, i.e., it is a direct summand of each module-finite extension ring. Hence, if $(\Sym V)^G$ is $F$-regular, then it is a direct summand of $\Sym V$. But then it has FFRT by \cite[Proposition~3.1.4]{Smith:VdB}.
\end{proof}

We next present a family of examples where $(\Sym V)^G$ is not $F$-regular or even $F$-pure, but has FFRT:

\begin{example}
Let $p$ be a prime integer, $V\colonequals\FFp^4$, and $G$ the subgroup of $\GL(V)$ generated by the matrices
\[
\begin{bmatrix}
1 & 0 & 1 & 0\\
0 & 1 & 0 & 1\\
0 & 0 & 1 & 0\\
0 & 0 & 0 & 1
\end{bmatrix},
\qquad
\begin{bmatrix}
1 & 0 & 0 & 1\\
0 & 1 & 0 & 0\\
0 & 0 & 1 & 0\\
0 & 0 & 0 & 1
\end{bmatrix},
\qquad
\begin{bmatrix}
1 & 0 & 0 & 0\\
0 & 1 & 1 & 0\\
0 & 0 & 1 & 0\\
0 & 0 & 0 & 1
\end{bmatrix}.
\]
It is readily seen that the matrices commute, and that the group $G$ has order $p^3$. Consider the action of $G$ on the polynomial ring $S\colonequals\Sym V=\FFp[x_1,x_2,x_3,x_4]$, where $x_1,x_2,x_3,x_4$ are viewed as the standard basis vectors in $V$. While $x_1$ and $x_2$ are fixed under the action, the orbits of $x_3$ and~$x_4$ respectively consist of all linear forms
\[
x_3+\alpha x_1+\gamma x_2 \qquad\text{ and }\qquad x_4+\beta x_1+\alpha x_2,
\]
where $\alpha,\beta,\gamma$ are in $\FFp$. Using Moore determinants as in~\cite[Chapter~1.3]{Goss}, the respective orbit products may be expressed as
\[
u\colonequals 
\frac{\det\begin{bmatrix}
x_1 & x_2 & x_3\\[3pt]
x_1^p & x_2^p & x_3^p\\[3pt]
x_1^{p^2} & x_2^{p^2} & x_3^{p^2}
\end{bmatrix}}
{\det\begin{bmatrix}
x_1 & x_2\\[2pt]
x_1^p & x_2^p
\end{bmatrix}}
\qquad\text{ and }\qquad
v\colonequals 
\frac{\det\begin{bmatrix}
x_1 & x_2 & x_4\\[3pt]
x_1^p & x_2^p & x_4^p\\[3pt]
x_1^{p^2} & x_2^{p^2} & x_4^{p^2}
\end{bmatrix}}
{\det\begin{bmatrix}
x_1 & x_2\\[2pt]
x_1^p & x_2^p
\end{bmatrix}}.
\]
In addition to these, it is readily seen that the polynomial $t\colonequals x_1x_4^p-x_1^px_4+x_2x_3^p-x_2^px_3$ is invariant. These provide us with a \emph{candidate} for the invariant ring, namely
\[
C\colonequals\FFp[x_1,x_2,t,u,v].
\]
Note that $S$ is integral over $C$ as $x_3$ and $x_4$ are, respectively, roots of the monic polynomials
\[
\prod_{\alpha,\gamma\,\in\FFp}(T+\alpha x_1+\gamma x_2)-u
\qquad\text{and}\qquad
\prod_{\beta,\alpha\,\in\FFp}(T+\beta x_1+\alpha x_2)-v
\]
that have coefficients in $C$. Using the first of these polynomials, one also sees that
\[
[\Frac(C)(x_3):\Frac(C)]\ \le\ p^2.
\]
Bearing in mind that $t\in C$, one then has $[\Frac(C)(x_3,x_4):\Frac(C)(x_3)]\le p$, and hence
\[
[\Frac(S):\Frac(C)]\ \le\ p^3.
\]
Since $C\subseteq S^G\subseteq S$ and $|G|=p^3$, it follows that $\Frac(C)=\Frac(S^G)$. To prove that $C=S^G$, it suffices to verify that $C$ is normal. Note that $C$ must be a hypersurface; we arrive at its defining equation as follows: One readily verifies the identity
\begin{multline*}
\det\begin{bmatrix}x_1 & x_2\\[2pt] x_1^p & x_2^p\end{bmatrix}
\left(
\det\begin{bmatrix}x_1 & x_4\\[2pt] x_1^p & x_4^p\end{bmatrix} + \det\begin{bmatrix}x_2 & x_3\\[2pt] x_2^p & x_3^p\end{bmatrix}
\right)^p\\
- x_1^p \det\begin{bmatrix} x_1 & x_2 & x_4\\[3pt] x_1^p & x_2^p & x_4^p\\[3pt] x_1^{p^2} & x_2^{p^2} & x_4^{p^2}\end{bmatrix}
- x_2^p \det\begin{bmatrix}x_1 & x_2 & x_3\\[3pt] x_1^p & x_2^p & x_3^p\\[3pt] x_1^{p^2} & x_2^{p^2} & x_3^{p^2}\end{bmatrix}\\
= \ \left(\det\begin{bmatrix}x_1 & x_2\\[2pt] x_1^p & x_2^p\end{bmatrix}\right)^p
\left(
\det\begin{bmatrix}x_1 & x_4\\[2pt] x_1^p & x_4^p\end{bmatrix} + \det\begin{bmatrix}x_2 & x_3\\[2pt] x_2^p & x_3^p\end{bmatrix}
\right),
\end{multline*}
which may be rewritten as
\[
t^p\det\begin{bmatrix}x_1 & x_2\\[2pt] x_1^p & x_2^p\end{bmatrix}\ -\ 
vx_1^p\det\begin{bmatrix}x_1 & x_2\\[2pt] x_1^p & x_2^p\end{bmatrix}\ -\
ux_2^p\det\begin{bmatrix}x_1 & x_2\\[2pt] x_1^p & x_2^p\end{bmatrix}
\ = \
t\left(\det\begin{bmatrix}x_1 & x_2\\[2pt] x_1^p & x_2^p\end{bmatrix}\right)^p.
\]
Dividing by the determinant that occurs on the left, one then has
\begin{equation}
\label{equation:nfp:hypersurface}
t^p-vx_1^p-ux_2^p\ =\ t(x_1x_2^p - x_1^px_2)^{p-1}.
\end{equation}
The Jacobian criterion shows that a hypersurface with~\eqref{equation:nfp:hypersurface} as its defining equation must be normal; it follows that $C$ is indeed a normal hypersurface, with defining equation~\eqref{equation:nfp:hypersurface}, and hence that $C$ is precisely the invariant ring $S^G$. Equation~\eqref{equation:nfp:hypersurface} shows that $S^G$ is not $F$-pure: $t$ is in the Frobenius closure of $(x_1,x_2)S^G$, though it does not belong to this ideal.

It remains to prove that the ring $C=S^G$ has FFRT. For this, note that after a change of variables, one has
\[
S^G\ \cong\ \FFp[x_1,x_2,t,\tilde{u},\tilde{v}]/(t^p-\tilde{v}x_1^p-\tilde{u}x_2^p).
\]
But then $S^G$ has FFRT by \cite[Observation~3.7, Theorem~3.10]{Shibuta1}: Set $A\colonequals\FFp[x_1,x_2,\tilde{u},\ \tilde{v}]$, and note that
\[
A\ \subseteq\ S^G\ \subseteq\ A^{1/p},
\]
where $A$ is a polynomial ring.
\end{example}

\section*{Acknowledgments}

Calculations with the computer algebra system \texttt{Magma}~\cite{Magma} were helpful in obtaining the presentation of the invariant ring in Remark~\ref{remark:presentation}. The authors are also deeply grateful to Professor Kei-ichi Watanabe for valuable discussions, and to the referee for useful suggestions.


\end{document}